\documentclass[11pt]{article}
\usepackage{graphicx} 
\usepackage{amsmath}
\usepackage{amsfonts}
\usepackage{amsthm}
\usepackage{forest}
\usepackage{tikz}
\usetikzlibrary{automata, positioning, arrows.meta}

\usepackage[left=2cm,right=2cm,top=2cm,bottom=2cm]{geometry}

\theoremstyle{plain}
\newtheorem{theorem}{Theorem}

\newtheorem{lemma}{Lemma}
\newtheorem{proposition}{Proposition}
\newtheorem{corollary}{Corollary}
\newtheorem{conjecture}{Conjecture}

\theoremstyle{definition}
\newtheorem{definition}{Definition}
\newtheorem{remark}{Remark}
\newtheorem{example}{Example}

\title{Group Theory of the Kolakoski Sequence}
\author{Noah MacAulay\footnote{Dalhousie University. email:usernameneeded@gmail.com}}

\begin{document}

\maketitle

\begin{abstract}
Run-length decoding is an operation on sequences in which a positive integer $a$ is replaced by a run(sequence of repeated elements) of length $a$. \emph{Iterated} run-length decodings applied to sequences with alphabets consisting of pairs of positive integers $\{p, q\}$ have attracted attention from mathematicians, most notably in their role defining the well-known \emph{Kolakoski sequence}. $n$-th-iterated run-length decodings are controlled by naturally associated permutation automata $A^{p,q}_n$. Here we study the transformation groups $\mathcal{K}^{p,q}_n$ of these automata. They are subgroups of the automorphism group of binary trees of depth $n+1$. They are naturally subgroups of(and likely equal to) a certain group $\mathcal{J}_n^{p,q}$ with an intricate recursive structure; their limit group is plausibly weakly regular branch. As an application we determine the number of maximal-length orbits of the automata given an arbitrary input sequence for odd $n$.

\end{abstract}

\section{Introduction}

Run-length encoding is a fundamental operation on sequences, which encodes an input sequence of elements in terms of the length and content of its runs(sequences of repeated elements). Its inverse operation is run-length \emph{decoding}, which maps sequences of pairs $(l, c)$ to runs of the content $c$ of length $l$. For sequences over a two-element alphabet $\{a, b\}$, we only need to specify the content of the first element of the output sequence, as runs of the two elements must alternate. For example, the run-length decoding of $(4, 1, 3, 2)$  beginning with $a$ is $(a, a, a, a, b, a, a, a, b, b)$.

Mathematicians have taken an interest in run-length decodings over alphabets of two positive integers $\{p, q\}$. As the output sequence also consists of positive integers, the decoding operation can be iterated, and this produces sequences with interesting properties. 

Most famously, the Kolakoski sequence\footnote{Also known as the Oldenburger-Kolakoski sequence\cite{oldenburger}}\cite{kolakoski} is the infinite sequence defined as the fixed-point of run-length decoding(or equivalently, run-length encoding) over $\{ 1, 2\}$ starting with 1. It begins:

$$ 1, 2, 2, 1, 1, 2, 1, 2, 2, 1, 2, 2, 1, 1, 2, 1, 1, 2, 2, 1, 2, 1, 1, ...  $$

In this paper we will be concerned with sequences which are the product of a finite number $n$ of iterations of run-length decoding over $\{p,q\}$, which we refer to as $(n, p,q)$-\emph{differentiable}. These are of interest both as models of the structure of fixed-point sequences --- as fixed-point sequences are in the range of $n$-th iterated decoding for all $n$ --- and in their own right. It turns out that all $(n,p,q)$-differentiable sequences are governed by a naturally associated automaton\footnote{This automaton was first described, to the best of my knowledge, by Shen\cite{shen}. A very closely related automaton was described by Chv{\'a}tal\cite{chvatal}} $A_n^{p,q}$. We now turn to describing this automaton.

\section{The Automaton $A_n^{p,q}$ }

To understand how the automaton arises, let's first consider the structure of an arbitrary $n$-fold iterated run-length decoding(which we will sometimes abbreviate as RLD) over the alphabet $\{p, q\}$. 

\begin{definition}
Let $\{p, q\}$ be a pair of positive integers. Let $x$ be an element of $\{p, q\}$ and $s_{1:m}$ a sequence of positive integers. Define the run-length decoding $RLD^{p,q}(x, s_{1:m})$ as the sequence of elements of $\{p, q\}$ beginning with $x$ with run-lengths $s_{1:m}$. 
\end{definition}

\begin{example}
Over $\{3, 4\}$, the run-length decoding of $(3, 2, 5, 1)$ starting with $4$ is 
$$RLD^{3,4}(4, (3, 2, 5, 1)) = (4, 4, 4, 3, 3, 4, 4, 4, 4, 4,4, 3)$$
\end{example}

\begin{definition}
Let $x_{1:n}$ be a sequence of $n$ elements of $\{p, q\}$ and $s_{1:m}$ be an input sequence. We define the $n$-iterated run-length decoding $RLD^{p,q}_n(x_{1:n}, s_{1:m})$ of $s_{1:m}$ recursively as follows: 
$$RLD^{p,q}_1 (x_1, s_{1:m}) = RLD^{p,q}(x_1, s_{1:m}) $$
$$RLD^{p,q}_n (x_{1:n}, s_{1:m}) = RLD^{p,q}(x_n, RLD^{p,q}_{n-1}(x_{1:n-1}, s_{1:m}))$$
\end{definition}

\begin{example}

We compute the 3-iterated run-length decoding $RLD^{3, 2}_3((2, 3, 2), (1, 2))$.

$$RLD^{3, 2}(2, (1, 2)) = (2, 3, 3)$$
$$RLD^{3,2}(3, (2, 3, 3)) = (3, 3, 2, 2, 2, 3, 3, 3,)$$ 
$$RLD^{3, 2}(2,(3, 3, 2, 2, 2, 3, 3, 3) ) = (2, 2, 2, 3, 3, 3, 2, 2, 3, 3, 2, 2, 3, 3, 3, 2, 2, 2, 3, 3, 3) $$
\end{example}

The structure of iterated RLDs may seem very complicated in general. However, there is a natural way of ``breaking them down'' into component pieces. Say we have an input sequence $s_{1:m}$ and start-element sequence $x_{1:n}$ . Observe that after proceeding partway through said input sequence, say to $s_{1:k}$, the output sequences at every level must alternate from the last element of $\{p, q\}$ in the sequence so far to its other element, giving a new vector $y_{1:n}$. We can then consider the iterated run-length decoding $RLD^{p,q}_n(y_{1:n}, s_{k+1:m})$ and concatenate its output to our existing output. The transition function from $x_{1:n}$ to $y_{1:n}$ given $s_{1:k}$ is the automaton.

If this seems cryptic, try looking at examples \ref{ex:automaton} and \ref{ex:automaton2} below. But first, let's define our terms more carefully.

\begin{definition}[Opp and OppEnd]
Given $\{p, q\}$, define $\textnormal{Opp}(p) = q$ and $\textnormal{Opp}(q) = p$.
Given a sequence of elements of $\{p,q\}$, define $\textnormal{OppEnd}(s_{1:m}) = \textnormal{Opp}(s_m)$. 
Example: For $\{3, 4\}$, $\textnormal{OppEnd}(3, 4, 3, 3, 4) = 3$.
\end{definition}

\begin{definition}[The Automaton $A^{p,q}_n$]
Given $p, q$ and $n$ we define an automaton $A^{p,q}_n$ with input alphabet $\{p, q\}$ with states of the form $x_{1:n}$, sequences of $n$ elements of $\{p, q\}$. The transition function on input $s$ is defined as follows:
$$(A(x_{1:n}, s))_i = \textnormal{OppEnd}(RLD^{p,q}_i(x_{1:i}, s))$$
\end{definition}

The key proposition that allows us to break down iterated run-length decodings in terms of $A^{p,q}_n$ is as follows:

\begin{proposition} \label{prop:automatonprop}
Given an input sequence $s_{1:m}$, for any $k \in \{1,...,m\}$ we have 
$$RLD^{p,q}_n(x_{1:n}, s_{1:m}) = RLD^{p,q}_n(x_{1:n}, s_{1:k}) \circ RLD^{p,q}_n(A(x_{1:n}, s_{1:k}), s_{k+1:m})$$
where $\circ$ denotes string concatenation.

\end{proposition}

By iterating proposition 1, we obtain 

\begin{corollary} \label{corr:automaton}
Given an input sequence $s_{1:m}$, we have 
$$RLD^{p,q}_n(x_{1:n}, s_{1:m}) = \prod_{i=1}^m RLD^{p,q}_n(y_i, s_i)$$
where $y_1 = x_{1:n}$ and $$y_i = A(y_{i-1}, s_{i-1})$$
for $1 < i \leq m$
\end{corollary}

Corollary \ref{corr:automaton} shows that we can understand arbitrary $n$-iterated run-length decodings in terms of walks through the state-space of $A^{p,q}_n$ together with the $2^{n+1}$ strings $RLD_n^{p,q}(x_{1:n},s)$. The transition beginning at state $x$ on input $s$ emits the sequence $RLD^{p,q}_n(x, s)$ and proceeds to the new state $y = A^{p,q}_n(x, s)$.\footnote{As an illustration of the power of this approach: Chv{\'a}tal\cite{chvatal} proved fairly tight bounds on the density of 1s in the Kolakoski sequence by bounding their density in a related graph.}

It's probably easiest to see the truth of Proposition \ref{prop:automatonprop} and Corollary \ref{corr:automaton} by looking at some examples.

\begin{example} \label{ex:automaton}

Consider the 3-iterated (1,2)-run-length decoding of (1, 2, 1, 2) with starting sequence (2, 1, 2): 

\[
\begin{tabular}{c c c c c c c c c c c c}
    \textit{1} &   & \textit{2} &  &  &  \textit{1}  &   &  &  & \textit{2}  &   & \\
    \underline{2} &   & \underline{1} &  & 1&  \underline{2}  &   &  &  & \underline{1}  & 1 & \\
    \underline{1} & 1 & \underline{2} &  & 1&  \underline{2}  &   & 2&  & \underline{1}  & 2 & \\
    \underline{2} & 1 & \underline{2} & 2& 1&  \underline{2}  & 2 & 1& 1& \underline{2}  & 1 & 1
\end{tabular}
\]

States of the automaton are underlined and the input sequence is italicized. Here we have \\ $A^{1,2}_3(212, 1) = (1,2,2)$, $A^{1,2}_3(122, 2) = (2,2,2)$ and  $A^{1,2}_3(222, 1) = (1,1,2)$. The overall run-length decoding \\ $RLD^{1,2}_3((2,1,2), (1,2,1,2))$ is equal to the concatenation of $RLD^{1,2}_3(212, 1)$, $RLD^{1,2}_3(122, 2)$, $RLD^{1,2}_3(222, 1)$, $RLD^{1,2}_3(112, 2)$.

\end{example}

\begin{example} \label{ex:automaton2}
    Here's the 3-iterated (1,3)-RLD of the input sequence
    (1, 3, 1, 1) with starting sequence (3, 1, 3).

    \[
\begin{tabular}{c c c c c c c c c c c c c c c c c}
    \textit{1}    &   &   &    \textit{3} &    &   &   &    &  &   & \textit{1}     &   &   &\textit{1}    &   & \\
    \underline{3} &   &   & \underline{1} &    &   & 1 & 1 &   &   & \underline{3}  &   &   & \underline{1}&   & \\
    \underline{1} & 1 & 1 & \underline{3} &    &   & 1 & 3 &   &   & \underline{1}  & 1 & 1 & \underline{3}&   &  \\
    \underline{3} & 1 & 3 & \underline{1} & 1  & 1 & 3 & 1 & 1 & 1 & \underline{3}  & 1 & 3 & \underline{1}& 1 & 1
\end{tabular}
\]

\end{example}

\begin{remark}
Notice that in example \ref{ex:automaton2}, every element of the state of the automaton switches on every input. It's fairly easy to see that this always occurs for automata $A^{p,q}_n$ with both $p$ and $q$ odd; while for automata with both $p$ and $q$ even, the state never changes. This makes the structure of $(n,p,q)$-differentiable sequences easy to understand in these cases. Hence in the remainder of this paper we will focus on pairs $\{p,q\}$ with $p + q$ odd. 
\end{remark}

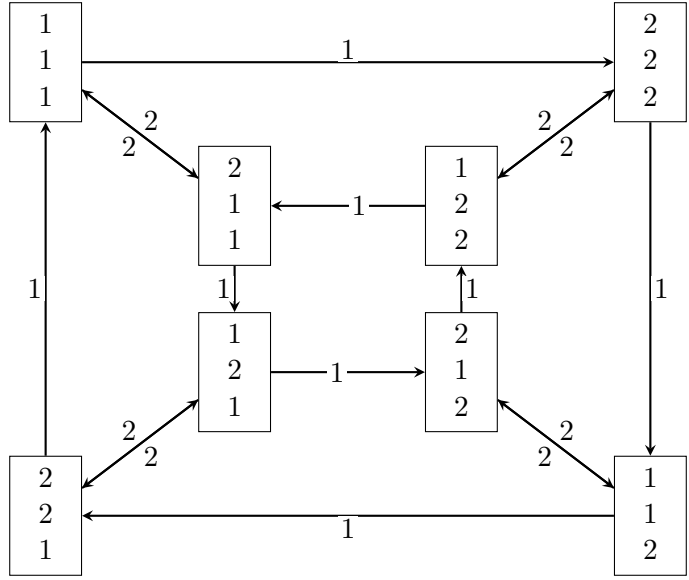
\begin{figure}[h]

\centering
\begin{tikzpicture}[
    >=stealth,
    node/.style={
        draw,
        rectangle,
        minimum width=0.95cm,
        minimum height=0.95cm,
        align=center,
        inner sep=2pt,
        font=\itshape
    },
    edge/.style={->, thick},
    lab/.style={font=\bfseries, fill=white, inner sep=1pt}
]

\node[node] (111) at (0,6) {$\begin{array}{c}1\\1\\1\end{array}$};
\node[node] (222) at (8,6) {$\begin{array}{c}2\\2\\2\end{array}$};
\node[node] (112) at (8,0) {$\begin{array}{c}1\\1\\2\end{array}$};
\node[node] (221) at (0,0) {$\begin{array}{c}2\\2\\1\end{array}$};

\node[node] (211) at (2.5,4.1) {$\begin{array}{c}2\\1\\1\end{array}$};
\node[node] (122) at (5.5,4.1) {$\begin{array}{c}1\\2\\2\end{array}$};
\node[node] (212) at (5.5,1.9) {$\begin{array}{c}2\\1\\2\end{array}$};
\node[node] (121) at (2.5,1.9) {$\begin{array}{c}1\\2\\1\end{array}$};

\draw[edge] (111) -- node[lab, above] {$1$} (222);
\draw[edge] (222) -- node[lab, right] {$1$} (112);
\draw[edge] (112) -- node[lab, below] {$1$} (221);
\draw[edge] (221) -- node[lab, left] {$1$} (111);

\draw[edge] (211) -- node[lab, left] {$1$} (121);
\draw[edge] (121) -- node[lab, left] {$1$} (212);
\draw[edge] (212) -- node[lab, right] {$1$} (122);
\draw[edge] (122) -- node[lab, right] {$1$} (211);

\draw[edge] (111) -- node[lab, above right] {$2$} (211);
\draw[edge] (211) -- node[lab, below left] {$2$} (111);

\draw[edge] (222) -- node[lab, above left] {$2$} (122);
\draw[edge] (122) -- node[lab, below right] {$2$} (222);

\draw[edge] (112) -- node[lab, below left] {$2$} (212);
\draw[edge] (212) -- node[lab, above right] {$2$} (112);

\draw[edge] (221) -- node[lab, above left] {$2$} (121);
\draw[edge] (121) -- node[lab, below right] {$2$} (221);

\end{tikzpicture}

\caption{Illustration of $A_3^{1,2}$. Edges are labeled by the input inducing that transition.}
\end{figure}

\vspace{1em}

Clearly the structure of $A^{p,q}_n$ is key for understanding $(n, p, q)$-differentiable sequences. The problem is that $A^{p,q}_n$ is itself rather inscrutable. As a first step forwards, we will rely on the following key property. Note that henceforth we will omit $p,q$ from $A^{p,q}_n$ when clear from context. 

\begin{lemma}
$A_n^{p,q}$ is a permutation automaton. Equivalently, the function $A_n^{p,q}(_, s)$ is invertible for all sequences $s$.
    
\end{lemma}

\begin{proof}
We will proceed by induction in $n$. Now for the base case we have

\[
A_1(x, s) = 
\begin{cases}
x, &  \text{if $|s|$ is even}, \\

(p+q) - x, & \text{if  $|s|$  is odd}
\end{cases}
\]
Clearly this is invertible.

Now, for the inductive step, we have 
\[
A_n(x, s) = 
\begin{cases}
(x_1,  A_{n-1}(x_{2:n}, RLD(x_1, s))),  &  \text{if $|s|$ is even}, \\

((p+q) - x_1,  A_{n-1}(x_{2:n}, RLD(x_1, s))) & \text{if  $|s|$  is odd}
\end{cases}
\]

Which can thus be inverted as follows(the inverse being taken in $x$):
\[
A_n^{-1}(x, s) = 
\begin{cases}
(x_1,  A_{n-1}^{-1}(x_{2:n}, RLD(x_1, s))),  &  \text{if $|s|$ is even}, \\

((p+q) - x_1,  A_{n-1}^{-1}(x_{2:n}, RLD((p + q) - x_1, s))) & \text{if  $|s|$  is odd}
\end{cases}
\]

\end{proof}

Now since $A_n^{p,q}$ is a permutation automaton, there is an associated permutation group consisting of all possible permutations of its states caused by an input sequence. We denote this group by $\mathcal{K}_n^{p,q}$. As it turns out, this group has a natural recursive structure which we turn to in the next section.

\section{The Transformation Groups $\mathcal{K}_n^{p,q}$ and $\mathcal{J}_n^{p,q}$}

We begin with an observation: The first $k$ digits of the image $A(x, s)$ are determined by the first $k$ digits of $x$. This means that transformations of the automaton's state can be represented in a tree: that is, does the first digit get flipped? Given that the first digit is $p$, does the second get flipped? Given that the first digit is $p$ and second $q$, does the third digit get flipped? etc.

In group-theoretic terms, this means that the transformation group $\mathcal{K}_n$ can be considered a subgroup of the $n$-th iterated wreath product of $\mathbb{Z}_2$, $ \mathbb{Z}_2 \wr ... \wr \mathbb{Z}_2$, equivalently, the group of automorphisms of a binary tree of depth $n$, $Aut(T_{n+1})$. We denote elements of this group by tuples $(f, A, B)$ where $f \in \mathbb{Z}_2$ and $A, B \in Aut(T_{n})$. The first element denotes if the given automorphism flips the root, and $A, B$ are the mappings of the left and right subtrees. Elements of this group can be represented as trees of depth $n$ with elements of $\mathbb{Z}_2$ at the roots and leaves. They have the multiplication rule:

$$(f, A, B) (g, C, D) =  \begin{cases}
    (f+g, AC, BD) & \text{if $f = 0$}\\
    (f+g, AD, BC) & \text{if $f = 1$}
     
\end{cases}$$

Now, as mentioned, $\mathcal{K}_n^{p,q}$ has a recursive structure, which manifests by it being a subgroup of another group $\mathcal{J}_n^{p,q}$. The structure of $\mathcal{J}_n^{p,q}$ is rather complicated, so to motivate it we first discuss it in informal terms in the next section.

\subsection{Informal Description}

Let's consider what elements of $\mathcal{K}_n^{p,q}$ represent. An element of the group is generated by a string $s$ of $p$s and $q$s. Its action on a string $x$ of length $n$ is determined by the parity of the length of its successive run-length-decodings on each prefix of the string -- that is, is the string $s$ of odd length? Is its $RLD$ with $x_1$ of odd length? Its $RLD$ with $x_1, x_2$? etc. We can arrange the possible $RLD$s in a tree; the parity of the length of the strings in this tree determine the group element. Here's an example with $s = 12$.

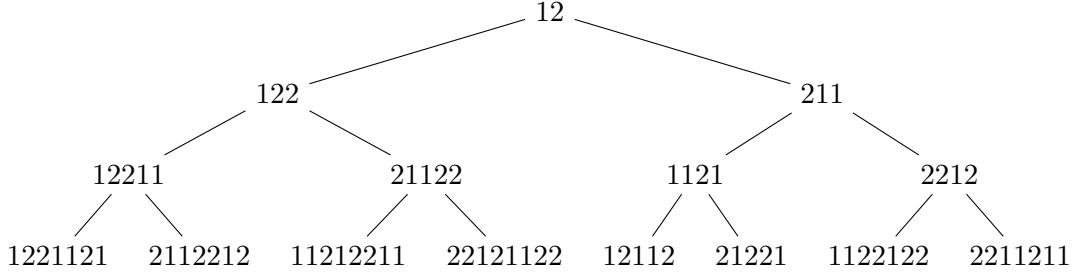
\begin{center}
\vspace{1em}
\begin{forest}
    [ 12
    [122  [12211 [1221121] [2112212] ][21122 [11212211] [22121122] ]]
    [211  [1121 [12112] [21221] ][2212 [1122122] [2211211] ]]
    ]
\end{forest}
\end{center}

This group element would send $1212 \rightarrow 1122$ for instance, as $|12| = 2$, $|122| = 3$, $|21122| = 5$ and $|11212211| = 8$.

What can we say about the structure of these trees? One simple observation is that the length-parity of immediate siblings is always equal, as they are both equal to the parity of the sum of elements in their parent. It turns out that considering the sums is useful in general. Here's the tree of sums of the above element:

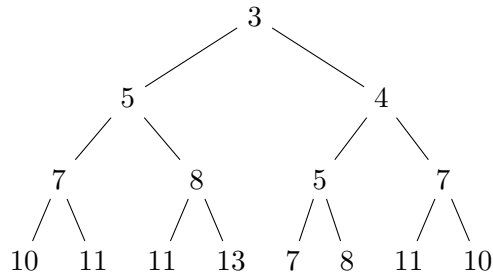
\begin{center}
\vspace{1em}
\begin{forest}
    [ 3
    [ 5  [7 [10] [11] ][8 [11] [13] ]]
    [ 4  [5 [7] [8] ][7 [11] [10] ]]
    ]
\end{forest}

\end{center}

There are certain relations that hold among sums at different parts of the tree. If we have a part of the tree as follows:

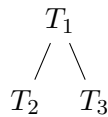
\begin{center}
\vspace{1em}
\begin{forest}
    [ $T_1$ [$T_2$] [$T_3$] ]
\end{forest}
\end{center}

Then we have the relation:

\begin{equation}\label{rel1}
\Sigma(T_1) = \frac{\Sigma(T_2) + \Sigma(T_3)}{p+q}
\end{equation}

where $\Sigma$ denotes the sum of elements in a string. This holds because any element $m$ of $T_1$ creates a run of $m$ $p$s in one child and $m$ $q$s in the other.

A more complex relation holds among the \emph{differences} between siblings. Consider the difference $\Sigma(T_2) - \Sigma(T_3)$. $T_2$ and $T_3$ are equal to each other under exchanging $p$s and $q$s. Thus any run of $m$ $p$s in $T_2$ contributes a total $(p-q)m$ to the difference. Now, the runs of $T_2$ and $T_3$ are controlled by the elements of $T_1$. An element $m$ of $T_1$ contributes a run of $m$ $p$s to $T_2$ and $m$ $q$s to $T_3$ if it is in an odd position in $T_1$; vice versa if it is in an even position. Thus we have the identity:

\[ \Sigma(T_2) - \Sigma(T_3) = (p-q)(\Sigma_{\text{odd}} (T_1) - \Sigma_{\text{even}} (T_1)) \]

Let's consider the quantity $\Sigma_{\text{odd}} (T_1) - \Sigma_{\text{even}} (T_1)$. Now, every element of $T_1$ is either $p$ or $q$, and every element is either in an odd or even position. Assume $|T_1|$ is even. This means that if there is a surplus of $p$s on odd positions -- say, $\alpha$ additional $p$s -- there must be an exactly opposite surplus of $q$s on even positions. Thus the sum can be rewritten as $(\alpha p - \alpha q)(p-q)$.
If $|T_1|$ is odd there will be an additional $p$ or $q$ at the end, contributing positively to the sum. Let's write this extra $p$ or $q$ as $v$. So in total the difference is $\delta(|T_1|  \text{odd})(p-q)v + (\alpha p - \alpha q)(p-q)$.\footnote{Here $\delta$ is 1 if the enclosed condition is true, 0 otherwise. Later we will somewhat abuse notation and sometimes use it to mean 1 if an enclosed element is the nontrivial flip}

This may seem like a pointless rearrangement. But its utility becomes apparent if we consider a depth-3 subtree:

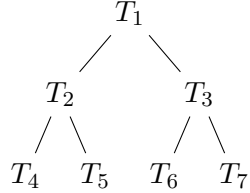
\begin{center}
\vspace{1em}
\begin{forest}
    [ $T_1$ 
    [$T_2$ [$T_4$] [$T_5$]] 
    [$T_3$ [$T_6$] [$T_7$]] ]
\end{forest}

\end{center}

From the above we have that $\Sigma(T_4) - \Sigma(T_5) = \delta(|T_2| \text{odd})(p-q) v + (\alpha p - \alpha q)(p - q)$, where $\alpha$ is the surplus of $p$s on odd positions of $T_2$. Similarly $\Sigma(T_6) - \Sigma(T_7) = \delta(|T_3| \text{odd})(p-q)w + (\beta p - \beta q)(p - q)$

Now consider that $T_3$ is simply $T_2$ with $p$ and $q$ interchanged. This means that a surplus of $p$s in $T_2$ becomes a surplus of $q$s in $T_3$, so $\beta = -\alpha$. Additionally, $|T_2| = |T_3|$, and the final element is swapped, so $v = (p+q) - w$. Thus we have the relation

\begin{equation}\label{rel2}
    \Sigma(T_4) - \Sigma(T_5) = \Sigma(T_7) - \Sigma(T_6) + \delta(|T_2| \text{odd})(p-q)(p+q)
\end{equation}

We can use this equation together with \ref{rel1} to derive 

\begin{equation}\label{rel3}
    \Sigma(T_1) = \frac{2(\Sigma(T_4) + \Sigma(T_6)) - \delta(|T_2| \text{odd})(p+q)(p-q)}{(p+q)^2}
\end{equation}

You may be wondering how these equations relate to the group structure. The group structure is entirely determined by the parity of the length of the sequences in the tree. Now, the parity of the length at one level is determined by the parity of the sum of the parent node. So merely from knowing the parity, we can derive the sum mod 2 of nodes immediately above. By using equation \ref{rel3} we can iterate this process: if we consider this equation mod $2^n$, its value is determined by the values of $\Sigma(T_4)$ and $\Sigma(T_6)$ mod $2^{n-1}$. So from knowledge of the sums mod $2^{n-1}$ at a given level, we can determine the sums mod $2^n$ two levels above. Thus from knowledge of the parities of the lengths of sequences at each node, we can infer the sums mod $2^{\lfloor n/2\rfloor}$ at height $n$. We can place the (sum, length-parity) tuples in a tree as shown below(the bottom row being only length-parities).

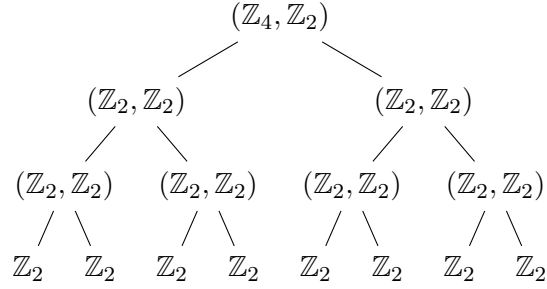
\begin{center}
\vspace{1em}
\begin{forest}
    [ {$(\mathbb{Z}_4, \mathbb{Z}_2)$}
    [{$(\mathbb{Z}_2, \mathbb{Z}_2)$}  [{$(\mathbb{Z}_2, \mathbb{Z}_2)$}  [{$\mathbb{Z}_2$}] [$\mathbb{Z}_2$] ][{$(\mathbb{Z}_2, \mathbb{Z}_2)$} [$\mathbb{Z}_2$] [$\mathbb{Z}_2$] ]]
    [{$(\mathbb{Z}_2, \mathbb{Z}_2)$}  [{$(\mathbb{Z}_2, \mathbb{Z}_2)$}  [{$\mathbb{Z}_2$}] [$\mathbb{Z}_2$] ][{$(\mathbb{Z}_2, \mathbb{Z}_2)$} [$\mathbb{Z}_2$] [$\mathbb{Z}_2$] ]]
    ]
\end{forest}
\end{center}

Thus, every group element has an associated tree in which the sums obey the relations \ref{rel1}, \ref{rel2}, \ref{rel3}, and length-parities are equal to the parity of their parent's sum. The set of elements of $Aut(T_{n+1})$ associated with these trees forms a group $\mathcal{J}_n^{p,q}$, of which $\mathcal{K}_n^{p,q}$ is a subgroup(and in fact, based on numerical evidence it seems extremely likely that $\mathcal{J}_n^{p,q} = \mathcal{K}_n^{p,q}$)
To prove this, we need to describe the group more formally.

\subsection{Group-Theoretic Description}

First, let's fix some notation. As mentioned, $\mathcal{K}_n^{p, q}$ and  $\mathcal{J}_n^{p, q}$ are subgroups of the nth-iterated wreath product of $\mathbb{Z}_2$, aka $Aut(T_{n+1})$. Given an element $g \in Aut(T_{n+1})$, we denote the left, right and root-flip sub-elements by $g_L$, $g_R$ and $g_f$ respectively. We may also chain these subscripts, denoting the left-right branch of a group element by $g_{LR}$, etc. We denote the dihedral group of order $2m$ by $D_{2m}$, and denote elements of this group by tuples $(f, x)$, where $f \in \mathbb{Z}_2$ and $x \in \mathbb{Z}_m$. We pick out elements of this tuple by $d_f$, $d_x$. We use the convention that $(f_1, x_1)(f_2, x_2) = (f_1+f_2, x_1+(-1)^{f_1}x_2).$ 

We can characterize the group $\mathcal{K}_n^{p, q}$ as follows. For the rest of this section, we assume that $p+q$ is odd.

\begin{definition}
The group $\mathcal{K}_n^{p,q}$ is the subgroup of $Aut(T_{n+1})$ generated by the following recursively defined elements

$$[p]_n = (1, [p]_{n-1}^p, [q]_{n-1}^p)$$
$$[q]_n = (1, [p]_{n-1}^q, [q]_{n-1}^q)$$

Where $[p]_1$ and $[q]_1$ are both defined as $(1)$ (the non-trivial element of $Aut(T_2)$)

\end{definition}

The definition of $\mathcal{J}_n^{p,q}$ is more complicated. We inductively define a series of groups $\mathcal{J}_n \subset Aut(T_{n+1})$,  maps $\psi: \mathcal{J}_n \rightarrow \mathbb{Z}_{2^{\lfloor n/2 \rfloor}}$ and $\Delta: \mathcal{J}_n \rightarrow D_{2^{\lfloor n/2 \rfloor+1}}$.

\begin{definition}
For n = 1, we define:

$\mathcal{J}_1^{p,q} = \mathbb{Z}_2$

$\psi_1(g) = 0$

$\Delta_1(g) = (g_f, 0)$

\end{definition}
\begin{definition}
    
For n = 2, we define:

$\mathcal{J}_2^{p,q}$ is the subgroup of elements $g \in Aut(T_3)$ such that $g_L = g_R$.

$\psi_2(g) = \delta(g_{Lf} = 1)$

$\Delta_2(g) = (g_f, \psi_2(g))$

\end{definition}

\begin{definition}
    Define the family of involutions of groups 
    $\Phi:D_{2m} \rightarrow D_{2m}$
    by 
    $$\Phi((f, x)) = (f, \delta(f=1) (p+q)(p-q) - x)$$
\end{definition}

\begin{definition}\label{group}

    For $n \geq 3$, we define $\mathcal{J}_n^{p,q}$ to be the set of elements $g \in Aut(T_{n+1})$ such that $g_L, g_R \in \mathcal{J}_{n-1}^{p,q}$, and

    $$\Delta_{n-1}(g_L) = \Phi(\Delta_{n-1}(g_R) )$$

\end{definition}

\begin{definition}
    For $n \geq 3$, we define $\psi_n:\mathcal{J}_n^{p,q} \rightarrow \mathbb{Z}_{2^{\lfloor n/2 \rfloor}}$ 
    by 
    
    $$\psi_n(g) = \begin{cases}

    \dfrac{\psi_{n-1}(g_L) + \psi_{n-1}(g_R)}{p+q} & \text{for odd $n$}\\
    
    \dfrac{2(\psi_{n-2}(g_{LL}) + \psi_{n-2}(g_{RL})) - \delta(g_{Lf}=1)(p+q)(p-q)}{(p+q)^2} & \text{for even $n$}

    \end{cases}
    $$
\end{definition}

\begin{definition}
    For $n \geq 3$, we define $\Delta_n:\mathcal{J}_n^{p,q} \rightarrow D_{2^{\lfloor n/2 \rfloor+1}}$ by 
    $$ \Delta_n(g) = \begin{cases} 

     (g_f, \psi_{n-1}(g_L) - \psi_{n-1}(g_R))& \text{for odd $n$} \\
    (g_f, (p+q)\psi_n(g) - 2\psi_{n-1}(g_R)) & \text{for even $n$}
    
    \end{cases}$$
\end{definition}

Note that given the above, definition \ref{group} is equivalent to the following statement for even $n$:
\begin{equation}\label{defin}
    \psi_{n-2}(g_{LL}) - \psi_{n-2}(g_{LR}) = \psi_{n-2}(g_{RR}) - \psi_{n-2}(g_{RL}) + \delta(g_{Lf}=1)(p+q)(p-q)
\end{equation}
together with the condition that $g_{Lf} = g_{Rf}$

Now we must verify that these definitions form a group:

\begin{theorem}
    Using the definitions above, we have
    
    (a) $\mathcal{J}_n ^{p,q}$ is a group
    
    (b) $\psi_n$ is a homomorphism

    (c) $\Delta_n$ is a homomorphism

\end{theorem}

\begin{proof}
    We proceed by induction. The claims are clear for $n=1,2$. 
    
    \vspace{1em}
    
    (a) Proof $\mathcal{J}_n^{p,q}$ is a group: If we have two elements $g, h   \in \mathcal{J}_n^{p,q}$, then  $gh$ is either equal to $(g_f+ h_f, g_Lh_L,g_Rh_R)$ or $(g_f+ h_f, g_Lh_R,g_Rh_L)$, depending on whether $g_f$ is 0 or 1, respectively. In the first case we have $\Phi(\Delta(g_L h_L)) = \Phi(\Delta(g_L)) \Phi(\Delta(h_L)) =\Delta(g_R h_R)$, in the second $\Phi(\Delta(g_L h_R)) = \Phi(\Delta(g_L)) \Phi(\Delta(h_R)) = \Delta(g_R h_L)$ (because $\Phi$ is an involution). So $gh$ is in $\mathcal{J}_n^{p,q}$

    As for inverses, if $\Delta(g_L) = \Phi(\Delta(g_R))$ we have $\Delta(g_L^{-1}) = \Phi(\Delta(g_R^{-1}))$ and $\Delta(g_R^{-1}) = \Phi(\Delta(g_L^{-1}))$, this is enough to show $g^{-1} \in \mathcal{J}_n^{p,q}$

    \vspace{1em}
    
    (b.i) Proof $\psi_n$ is a homomorphism, $n$ odd: $$\psi_n (gh) = \frac{\psi_{n-1}((gh)_L) + \psi_{n-1}((gh)_R)}{p+q} = \frac{\psi_{n-1}(g_L) + \psi_{n-1}(g_R) + \psi_{n-1}(h_L) + \psi_{n-1}(h_R)}{p+q} = \psi_n(g) + \psi_n(h)$$

    (b.ii) Proof $\psi_n$ is a homomorphism, $n$ even. We consider two further sub-cases: $g_{Lf} = 0$ and $g_{Lf} = 1$
    
    (b.ii.i) $g_{Lf} = 0$. In this case we have

    $$\psi_n(gh) = \frac{2(\psi_{n-2}((gh)_{LL}) + \psi_{n-2}((gh)_{RL})) - \delta(h_{Lf})(p+q)(p-q) }{(p+q)^2}$$

    $$ = \frac{2(\psi_{n-2}(g_{LL}) + \psi_{n-2}(h_{LL}) +  \psi_{n-2}(g_{RL}) + \psi_{n-2}(h_{RL})) - \delta(h_{Lf})(p+q)(p-q) }{(p+q)^2}$$

    $$ = \psi_n(g) + \psi_n(h)$$

    (b.ii.ii) $g_{Lf} = 1$. In this case we have
    $$\psi_n(gh) = \frac{2(\psi_{n-2}((gh)_{LL}) + \psi_{n-2}((gh)_{RL})) - (1 - \delta(h_{Lf}))(p+q)(p-q) }{(p+q)^2}$$
    
    $$ = \frac{2(\psi_{n-2}(g_{LL}) + \psi_{n-2}(h_{LR}) +  \psi_{n-2}(g_{RL}) + \psi_{n-2}(h_{RR})) - (1 - \delta(h_{Lf}))(p+q)(p-q) }{(p+q)^2}$$
    By equation \ref{defin}, this is 
    $$  \frac{2(\psi_{n-2}(g_{LL}) + \psi_{n-2}(h_{LL}) +  \psi_{n-2}(g_{RL}) + \psi_{n-2}(h_{RL}) - \delta(h_{Lf})(p+q)(p-q)) - (1 - \delta(h_{Lf}))(p+q)(p-q)) }{(p+q)^2}$$

    $$ = \frac{2(\psi_{n-2}(g_{LL}) + \psi_{n-2}(h_{LL}) +  \psi_{n-2}(g_{RL}) + \psi_{n-2}(h_{RL})) - (1 + \delta(h_{Lf}))(p+q)(p-q)) }{(p+q)^2}$$

    $$= \psi_n(g) + \psi_n(h)$$
    
    \vspace{1em}

    (c.i) Proof $\Delta_n$ is a homomorphism. We consider subcases $n$ odd or even, and $g_f = 0$ or $1$. 

    (c.i.i) $n$ odd, $g_f = 0$. We have
    $$\Delta_n(gh) = ((gh)_f, \psi_{n-1}((gh)_L) - \psi_{n-1}((gh)_R))$$
    $$ = (h_f, \psi_{n-1}(g_L) - \psi_{n-1}(g_R) + \psi_{n-1}(h_L)  - \psi_{n-1}(h_R)) = \Delta_n(g)\Delta_n(h)$$

    (c.i.ii) $n$ odd, $g_f = 1$. We have 
    $$\Delta_n(gh) = ((gh)_f, \psi_{n-1}((gh)_L) - \psi_{n-1}((gh)_R))$$
    $$ = (1 + h_f, \psi_{n-1}(g_L) - \psi_{n-1}(g_R) - (\psi_{n-1}(h_L) - \psi_{n-1}(h_R)))$$
    $$ = \Delta_n(g)\Delta_n(h)$$

    (c.ii.ii). $n$ even, $g_f = 0$. We have 
    $$\Delta_n(gh) = ((gh)_f, (p+q)\psi_{n}(gh) - 2\psi_{n-1}((gh)_R))$$
    $$= (h_f, (p+q)(\psi_{n}(g) + \psi_n(h)) - 2(\psi_{n-1}(g_R) + \psi_{n-1}(h_R)))$$
    $$= \Delta_n(g)\Delta_n(h)$$

    (c.ii.ii) $n$ even, $g_f = 1$. We have
    $$\Delta_n(gh) = ((gh)_f, (p+q)\psi_{n}(gh) - 2\psi_{n-1}((gh)_R))$$
    \begin{equation}\label{lala} = (1+h_f, (p+q)(\psi_{n}(g) + \psi_n(h)) - 2(\psi_{n-1}(g_R) + \psi_{n-1}(h_L)))\end{equation}
    Now consider the sum $(p+q)\psi_n(h) - 2\psi_{n-1}(h_L)$. As $n$ is even this is equal to 
    $$\frac{2(\psi_{n-2}(h_{LL}) + \psi_{n-2}(h_{RL})) - \delta(h_{Lf})(p+q)(p-q)}{p+q} - 2\frac{\psi_{n-2}(h_{LL}) + \psi_{n-2}(h_{LR})}{p+q}$$
    $$=\frac{2(\psi_{n-2}(h_{RL}) - \psi_{n-2}(h_{LR})) - \delta(h_{Lf})(p+q)(p-q)}{p+q}$$
    By equation \ref{defin}, this is equal to
    $$\frac{2(\psi_{n-2}(h_{RR}) - \psi_{n-2}(h_{LL})) + \delta(h_{Lf})(p+q)(p-q)}{p+q}$$
    $$ = 2\psi_{n-1}(h_R) - (p+q)\psi_n(h)$$
    Substituting this into \ref{lala} we have 
    $$\Delta_n(gh) = (1+h_f, (p+q)(\psi_n(g) - \psi_n(h)) - 2(\psi_{n-1}(g_R) - \psi_{n-1}(h_R)))$$
    $$ = \Delta_n(g)\Delta_n(h)$$

\end{proof}

The following lemma is useful:

\begin{lemma} \label{mod}

For all $n \geq 3$, we have $ \psi_n(g) = (\psi_{n-1}(g_L) + \psi_{n-1}(g_R))/(p+q) \mod{2^{\lfloor (n-1)/2 \rfloor}}$
    
\end{lemma}

\begin{proof}
    For odd $n$, this follows trivially from the definition of $\psi_n$. For even $n$, by \ref{defin} we have 
    $$\psi_n(g) = \frac{\psi_{n-2}(g_{LL}) + \psi_{n-2}(g_{LR}) +\psi_{n-2}(g_{RL}) +\psi_{n-2}(g_{RR})}{(p+q)^2} \mod 2^{\lfloor (n-1)/2 \rfloor}$$
    $$ = \frac{\psi_{n-1}(g_L) + \psi_{n-1}(g_{R})}{(p+q)} \mod 2^{\lfloor (n-1)/2 \rfloor}$$
\end{proof}

\subsection{$\mathcal{K}_n \subset \mathcal{J}_n$}

Having verified that the definition of $\mathcal{J}_n$ is well-formed, proving $\mathcal{K}_n$ is a subgroup is relatively easy. For this, it is enough to prove $[p]_n, [q]_n \in \mathcal{J}_n$.

For the purposes of the proof, is it convenient to introduce an element $y_n = (0, [p]_{n-1}, [q]_{n-1})$. Note that $[p]_n = y^p f$ and $[q]_n = y^q f$ (where we abuse notation somewhat by denoting by $f$ the element of $\mathcal{J}_n$ flipping the root).

\begin{theorem}

For all $n > 0$, we have 

(a) $y_n, [p]_n, [q]_n \in \mathcal{J}_n$

(b) $\psi_n(y_n) = 1$, $\psi_n([p]_n) = p$, $\psi_n([q]_n) = q$

(c) $\Delta_n(y_n) = (0,(p-q))$, $\Delta_n([p]_n) = (1, p(p-q))$, $\Delta_n([q]_n) = (1,q(p-q))$.

\end{theorem}

\begin{proof}

We again proceed by induction. The claims are fairly obvious for $n=1, 2$.

(a) By (c) $\Delta_{n-1}([p]_{n-1}) = (1, p(p-q))$. $\Phi((1, p(p-q))) = (1, q(p-q)) = \Delta_{n-1}([q]_{n-1}) $. Thus $y_n \in \mathcal{J}_n$. Since $f \in \mathcal{J}_n$, this implies that $[p]_n = y^p f$ and $[q]_n = y^q f$ are also.

(b) For odd $n$, we have that $\psi_n(y_n) = (\psi_{n-1}([p]_{n-1}) + \psi_{n-1}([q]_{n-1}))/(p+q) = 1$. 

For even $n$, $$\psi_n(y_n) = \frac{2(\psi_{n-2}([p]_{n-2}^p) +\psi_{n-2}([p]_{n-2}^q)) - \delta([p]_{(n-1)f})(p+q)(p-q)}{(p+q)^2}$$
$$=\frac{2(p^2 + pq) - (p+q)(p-q)}{(p+q)^2} = 1 $$

This implies $\psi_n([p]_n) = \psi_n(y)^p\psi_n(f) = p$ and $\psi_n([q]_n) = \psi_n(y)^q\psi_n(f) = q$.

(c) For odd $n$, $\Delta_n(y_n) = (0, \psi_{n-1}([p]_{n-1}) - \psi_{n-1}([q]_{n-1})) = (0, p-q)$
For even $n$, $\Delta_n(y_n) = (0, (p+q)\psi_n(y_n) - 2\psi_{n-1}([q]_{n-1})) = (0, p-q)$.
This implies $\Delta_n([p]_n) = \Delta_n(y_n)^p\Delta_n(f) = (0,p(p-q))(1,0)=(1,p(p-q))$ and $\Delta_n([q]_n) = \Delta_n(y_n)^q\Delta_n(f) = (0,q(p-q))(1,0)=(1,q(p-q))$

\end{proof}

Now, does the reverse inclusion hold? Based on evidence in the next section, it seems very likely that it does. Thus we venture the following:

\begin{conjecture}\label{conj}
    $\mathcal{J}_n^{p,q} = \mathcal{K}_n^{p,q}$
\end{conjecture}
    
\subsection{Order of $\mathcal{J}_n$}

The order of $\mathcal{J}_n$ can be computed. Comparing to finite $n$ computations of $\mathcal{K}_n^{p,q}$ provides numerical evidence for conjecture \ref{conj}.

\begin{lemma}
$\Delta_{n}$ is surjective on $D_{2^{\lfloor n/2 \rfloor + 1}}$
\end{lemma}

\begin{proof}
    $\Delta_n(y_n) = (0,p-q)$. $p-q$ is odd, so this generates the subgroup $(0,x)$. And $\Delta_n(f) = (1,0)$. This is enough to generate $D_{2^{\lfloor n/2 \rfloor 1}}$.
\end{proof}

\begin{lemma}
    $|\mathcal{J}_n| = 2|\mathcal{J}_{n-1}|^2/|D_{2^{\lfloor(n-1)/2 \rfloor+1}}|$
\end{lemma}

\begin{proof}
Consider an element $g$ of $\mathcal{J}_n$. Given the definition of $\mathcal{J}$, we may enumerate elements by firstly choosing an element $g_f \in \mathbb{Z}_2$, then an element $g_L$, then an element $g_R$ subject to $\Phi(\Delta_{n-1}(g_L)) = \Delta_{n-1}(g_R)$. $g_f$ and $g_L$ can be chosen freely, contributing 2 and $|\mathcal{J}_{n-1}|$ to the product, while there are  $|\mathcal{J}_{n-1}|/|D_{2^{\lfloor(n-1)/2 \rfloor+1}}|$ choices for $g_R$ conditional on $g_L$(since $\Delta_{n-1}(\mathcal{J}_{n-1}) = D_{2^{\lfloor(n-1)/2 \rfloor+1}}$). This implies the lemma.

\end{proof}

\begin{corollary}
    $|\mathcal{J}_n| = 2^{\lceil(2^n+3n)/6\rceil}$. Equivalently, $|\mathcal{J}_n| = 2^{g(n)}$, where $$g(n) = \begin{cases}
        \frac{2^n + 3n + 1}{6} & \text{for odd $n$} \\
        \frac{2^n + 3n + 2}{6} & \text{for even $n$}
    \end{cases}$$
    \end{corollary}

\begin{proof} By induction. The claim is clear for $n = 1$. 

Case: $n$ even. By the lemma we have $g(n) = 2 \times g(n-1)+1-(n/2) = \frac{2^n + 6(n-1) + 2}{6} + 1 - \frac{n}{2} = \frac{2^n + 3n + 2}{6}$

Case: $n$ odd. By the lemma we have $g(n) = 2 \times g(n-1) + 1 - ((n+1)/2) = \frac{2^n + 6(n-1) + 4}{6} - \frac{n-1}{2} = \frac{2^n + 3n + 1}{6}$
    
\end{proof}

We can compare the orders of $\mathcal{K}_n^{p,q}$ using computer algebra systems(we used GAP). We computed $|\mathcal{K}_n^{1,2}|$ up to $n=12$. Here are the results:

\vspace{1em}
\begin{tabular}{| c c c c c c c c c c c c c |}
\hline
$n$ & 1 & 2 & 3 & 4 & 5 & 6 & 7 & 8 & 9 & 10 & 11 & 12 \\
\hline
$\log_2|\mathcal{J}_n|$ & 1 & 2 & 3 & 5 & 8 & 14 & 25 & 47 & 90 & 176 & 347 & 689 \\
\hline
$\log_2|\mathcal{K}_n^{1,2}|$ & 1 & 2 & 3 & 5 & 8 & 14 & 25 & 47 & 90 & 176 & 347 & 689 \\
\hline
\end{tabular}

As can be seen, the agreement is exact. Identical results were obtained for $\mathcal{K}_n^{3,2}$. Based on this, it seems quite plausible that $\mathcal{J}_n^{p,q} = \mathcal{K}_n^{p,q}$.

\subsection{Limit Groups}

We briefly describe a family of infinite groups $\mathcal{J}_\infty^{p,q} \subset Aut(T_\infty)$. Here $Aut(T_\infty)$ is the automorphism group of the infinite binary tree.
It seems likely that these groups are the inverse limits of the groups $\mathcal{J}_n^{p,q}$, although we have not formally verified this. We denote the 2-adic integers by $\widehat{\mathbb{Z}_2}$, and the dihedral 2-adic integers by $Dih(\widehat{\mathbb{Z}_2})$ (meaning by this the semidirect product of $\mathbb{Z}_2$ and $\widehat{\mathbb{Z}_2}$ with $\mathbb{Z}_2$ acting by inversion). We specify elements of $Aut(T_\infty)$ by functions $f:T_\infty \rightarrow\mathbb{Z}_2$.

\begin{definition}
    Define the involution of groups 
    $\Phi:Dih(\widehat{\mathbb{Z}_2}) \rightarrow Dih(\widehat{\mathbb{Z}_2})$
    by 
    $$\Phi((f, x)) = (f, \delta(f) (p+q)(p-q) - x)$$
\end{definition}

\pagebreak

\begin{definition}
An element $f \in Aut(T_\infty)$ is part of $\mathcal{J}_\infty^{p,q}$ iff there exist functions $\psi:T_\infty \rightarrow \widehat{\mathbb{Z}_2}$ and $\Delta: T_\infty \rightarrow Dih(\widehat{\mathbb{Z}_2})$ satisfying:

$$\psi(x) = \frac{\psi(x_L) + \psi(x_R)}{p+q}$$

$$\Delta(x) = (f(x), \psi(x_L) - \psi(x_R))$$

$$\Delta(x_L) = \Phi(\Delta(x_R))$$

for all $x \in T_\infty$

\end{definition}

$\mathcal{J}_\infty^{p,q}$ seems to be a weakly regular branch group in the sense of \cite{branch}.

\section{Orbits of $A_n^{p,q}$}

As an application of these ideas, here we derive a simple criterion determining the number of orbits of maximal length of an arbitrary input sequence to $A_n^{p,q}$ for odd $n$. This partially solves a conjecture of Shen\cite{shen}, who conjectured that $A_n^{1,2m}$ always has such an orbit with input 1 passing through $(1, 1, ...1)$. Here we prove that a maximal orbit exists, but not that it passes through $(1, 1,...,1)$.

\begin{lemma}
    All orbits of $A_n^{p,q}(-,s)$ have lengths that are powers of 2
\end{lemma}
\begin{proof}
By induction. The claim is trivial for $n=1$.

Now, for an element $x_{1:n}$ to be in an orbit, that is $A_n(x, s^l) = x$ it must be the case that $x_{1:n-1}$ is sent to itself by $s^l$ also. Let $v$ be the smallest integer for which this holds; by induction it is a power of 2. Now, if $RLD_{n-1}(x_{1:n-1}, s^v)$ is of even length then $x$ is itself in an orbit of length $v$, a power of 2. If it is of odd length, then $x$ is in an orbit of length $2v$, a power of 2.

\end{proof}

The following lemma is convenient:

\begin{lemma}\label{obvious}

An element $s$ has $m$ cycles of length $2^v$, $v>0$ iff $s^2$ has $2m$ cycles of length $2^{v-1}$. Further, it suffices to consider cycles in $s^2_{LL}$, $s^2_{LR}$, $s^2_{RL}$, $s^2_{RR}$.
        
\end{lemma}

\begin{proof}

    The first part of the statement is obvious.

    The second follows because, since $s_{Lf} = s_{Rf}$, we must have that $(s^2)_f = (s^2)_{Lf} = (s^2)_{Rf} = 0$. Therefore we can break up the action of $s^2$ into the four independent pieces $s^2_{LL}$, $s^2_{RL}$, $s^2_{LR}$, $s^2_{RR}$.  
\end{proof}

This lets us analyze orbits of $\mathcal{K}_n^{p,q}$ in terms of $\mathcal{K}_{n-2}^{p,q}$.

\begin{corollary}

The length of an orbit of $A_n^{p,q}(-,s)$ is upper bounded by $ 2^{\lceil n/2\rceil}$
    
\end{corollary}

\begin{proof}

We proceed by induction. The claim is obvious for $n = 1,2$. 

Then for larger $n$, we can consider $s^2_{LL}$, etc., which have cycles of length at most  $ 2^{\lceil n/2\rceil -1}$ by induction. Ergo $s$ has cycles of length at most $ 2^{\lceil n/2\rceil}$
    
\end{proof}

The following lemma is a special case of lemma \ref{mod}: 
\begin{lemma}
    For $n \geq 3$, $\psi_n(g) = \psi_{n-1}(g_L) + \psi_{n-1}(g_R) \mod 2$
\end{lemma}

\begin{theorem}
    For odd $n \geq 3$, $A_n^{p,q}(-,s)$ has exactly two orbits of length $2^{\lceil n/2 \rceil}$ iff $|s|$ is odd and $\Sigma(s)$ is odd; otherwise, it has none.
\end{theorem}

\begin{proof}

    On a group level this is equivalent to the statement $s$ has two orbits of length $2^{\lceil n/2 \rceil}$ iff $\psi(s) = 1 \mod 2$ and $s_f = 1$.
    
    As usual we proceed by induction. Consider the case $n=3$. We can represent an element $s$ of $\mathcal{K}_3^{p,q}$  like:
    
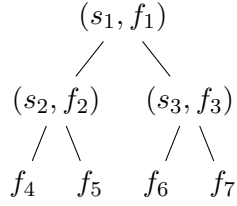
\begin{center}
    \begin{forest}
    [ {$(s_1, f_1)$} 
    [{$(s_2, f_2)$} [{$f_4$}] [{$f_5$}]] 
    [{$(s_3,f_3)$} [$f_6$] [$f_7$]] ]
    \end{forest}
    \end{center}

    Here $s_1, s_2, s_3 \in \mathbb{Z}_2$ are the $\psi(s)$, $\psi(s_L)$, etc. The $f_i \in \mathbb{Z}_2$ are the $s_f, s_{Lf}, s_{LLf}$, etc.

    In order for $s$ to have an orbit of length 4, $s^2$ must have an orbit of length 2. For this to be the case we must have $f_5 \ne f_6$ and $f_1 = 1$. $f_5$ and $f_6$ are controlled by $s_2$ and $s_3$, so we must have $s_2 \ne s_3$, which is the case iff $s_1 = 1$. So $s_1 = 1, f_1 = 1$ are necessary and sufficient. Further, if this holds we will have $(s^2)_{LLf} = (s^2)_{LRf} = (s^2)_{RLf} = (s^2)_{RRf} = 1$.
    So $s^2$ will have 4 orbits of length 2, ergo $s$ will have 2 orbits of length 4.
    
    Now for the inductive step, consider an element $s \in \mathcal{K}_n^{p,q}$, $n$ odd. We can represent it as follows:

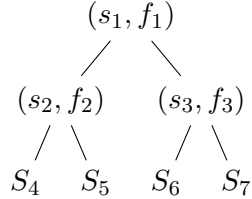
\begin{center}
\begin{forest}
    [ {$(s_1, f_1)$} 
    [{$(s_2, f_2)$} [{$S_4$}] [{$S_5$}]] 
    [{$(s_3,f_3)$} [$S_6$] [$S_7$]] ]
\end{forest}
\end{center}

Where $s_1, s_2, s_3 \in \mathbb{Z}_{2^{\lfloor n/2 \rfloor}}$, $f_1, f_2, f_3 \in \mathbb{Z}_2$, and $S_4, S_5, S_6, S_7 \in \mathcal{K}_{n-2}^{p,q}$

Now, for $s$ to have an orbit of length $2^{\lceil n/2 \rceil}$ we need one of $s^2_{LL}, s^2_{LR}, s^2_{RL}, s^2_{RR}$
to have an orbit of length $2^{\lceil n/2 \rceil-1}$, which requires $f=1$ and odd $\psi$. Now, $S_{4f} = S_{5f}$ and $S_{6f} = S_{7f}$, so in order for one of the $\{s^2_{LL},...\}$ to have $f=1$, we must have $f_1 = 1$. Further, we must \emph{not} have $S_{4f} = S_{6f}$. $S_{4f}$ and $S_{6f}$ are controlled by the parity of $s_2$ and $s_3$ respectively, so this is equivalent to $s_2$ and $s_3$ having differing parity, which is in turn equivalent to $s_1$ having odd parity.

So this means $s_1$ odd, $f_1=1$ is necessary. But is it sufficient? Yes: $s_1$ odd, $f_1=1$ imply that $\psi (s^2_L) = \psi(s^2_R)$ are both odd. This in turn implies that $(s^2)_{LLf} = (s^2)_{LRf} =(s^2)_{RLf}=(s^2)_{RRf}=1$. But it also implies that $\psi((s^2)_{LR})$ and  $\psi((s^2)_{LL})$ have different parities, and $\psi((s^2)_{RL})$ and $\psi((s^2)_{RR})$ have different parities. All together, this implies that 2 of $s^2_{LL}, s^2_{LR}, s^2_{RL}, s^2_{RR}$ have $\psi$ odd and $f=1$, meaning that 2 of them have 2 orbits of length $2^{\lceil n/2 \rceil-1}$. So there are 4 such orbits of $s^2$ in total, so $s$ has 2 orbits of length $2^{\lceil n/2 \rceil}$

\end{proof}

\section{Acknowledgements}
I would like to thank Bobby Shen for his paper \cite{shen} which inspired me to look for group-theoretic structure in $A_n^{p,q}$. Although all the main concepts and theorems in this paper are due to me, AI assistants proved very useful in my investigations. In particular GPT5 helped greatly with GAP/pysage/LaTeX coding, literature search, and discussion of the mathematics. The exact formula for $|\mathcal{J}_n|$ is due to Kimi2.6.


\begin{thebibliography}{9}

\bibitem{oldenburger} Rufus Oldenburger, Exponent trajectories in symbolic dynamics, Trans. Amer. Math. Soc., Vol. 46 (1939), pp. 453-466.

\bibitem{kolakoski} W. Kolakoski, Self Generating Runs, Problem 5304, American Math. Monthly 72 (1965), 674. Solution: American Math. Monthly 73 (1966), 681--682.

\bibitem{shen}
Shen, Bobby. ``The Kolakoski sequence and related conjectures about orbits." Experimental Mathematics 29.1 (2020): 54-65.

\bibitem{chvatal}
Chvátal, Vašek. ``Notes on the Kolakoski sequence." Rapport, DIMACS Techn. Rep (1994).

\bibitem{branch}
Bartholdi, Laurent, Rostislav I. Grigorchuk, and Zoran Šuni. ``Branch groups." Handbook of algebra. Vol. 3. North-Holland, 2003. 989-1112.

\end{thebibliography}
\end{document}